\documentclass{article}
\usepackage{amsmath}
\usepackage{amsthm}
\usepackage{amsfonts}
\usepackage{tcolorbox}
\usepackage{graphicx}
\usepackage{pgfplots}
\usepackage{tikz}
\usepackage{authblk}
\usepackage{tikz-network}

\title{Spectral convergence of random regular graphs: Chebyshev polynomials, non-backtracking walks, and unitary-color extensions}
\author{}

    \author[1]{Yulin Gong}\author[2]{Wenbo Li}\author[3]{Shiping Liu}
\affil[1]{Department of Mathematical Sciences, Tsinghua University, Beijing 100084, China}
\affil[2,3]{School of Mathematical Sciences, University of Science and Technology of China, Hefei 230026, China}
\affil[1]{gongyl22@mails.tsinghua.edu.cn}
\affil[2]{patlee@mail.ustc.edu.cn}
\affil[3]{spliu@ustc.edu.cn}

\date{}
\begin{document}
\maketitle

\newtheorem{theorem}{Theorem}[section]
\newtheorem{proposition}{Proposition}[section]
\newtheorem{corollary}{Corollary}[section]
\newtheorem{lemma}{Lemma}[section]
\newtheorem{definition}{Definition}[section]
\newtheorem{remark}{Remark}[section]
\newtheorem{eg}{Example}[section]
\newtheorem{Notation}{Notation}[section]
\begin{abstract}

%The study of the spectrum of regular graphs using Chebyshev polynomials and non-backtracking walks can be traced back to Friedman’s work and was further developed to analyze the convergence of spectral measures, as seen in the works of Serre and Sodin. 

In this paper, we give a short proof of the weak convergence to the Kesten-McKay distribution for the normalized spectral measures of random $N$-lifts. This result is derived by generalizing a formula of Friedman involving Chebyshev polynomials and non-backtracking walks. We also extend a criterion of Sodin on the convergence of graph spectral measures to regular graphs of growing degree. As a result, we show that for a sequence of random $(q_n+1)$-regular graphs $G_n$ with $n$ vertices, if $q_n = n^{o(1)}$ and $q_n$ tends to infinity, the normalized spectral measure converges almost surely in $p$-Wasserstein distance to the semicircle distribution for any $p \in [1, \infty)$. This strengthens a result of Dumitriu and Pal. Many of the results are extended to unitary-colored regular graphs.

\end{abstract}

\section{Introduction}

A sequence of regular graphs $G_n=(V_n,E_n)$ with fixed vertex degree and $|V_n|$ growing to infinity tends to be sparse due to the linear growth of edge numbers. Generically, $G_n$ contains very few number of circles per vertex as $n$ tends to infinity. 
This \emph{locally tree-like} structure has been noticed since the work of Bollob\'as \cite{Bollobas}
and Wormald \cite{Wormald}. For a comprehensive introduction to more related results, see \cite[Section 1]{Bauerschmidt-Huang-Yau} and the references therein. One classic result tells that the locally tree-like condition guarantees the weak convergence of the spectral measure of $G_n$ to the Kesten-McKay distribution as $n$ tends to infinity \cite{McKay}.

The \emph{locally tree-like} structure can be quantitatively understood by a series of polynomials modified from the Chebyshev polynomials as noticed by Friedman \cite[Lemma 3.3]{Friedman91}, see Lemma \ref{ArandXrq} for details. These polynomials has been employed to study the convergence of spectral measures of regular graphs, for example in \cite{Serre} and \cite{Sodin}. More precisely, for a $(q+1)\text{-}$regular graph $G=(V,E)$ with adjacency matrix $A$, the normalized spectral measure of $G$ is defined as \begin{equation}
    \mu(G):=\frac{1}{|V|}\sum_{1\leq k\leq |V|}\delta_{ q^{-1/2}\lambda_k(A)},
\end{equation}
where $\lambda_{|V|}(A) \leq \cdots \leq \lambda_2(A) \leq \lambda_1(A)$ are the eigenvalues of $A$. All information of the graph spectrum is then hidden in the integration of polynomials against this measure.
The usual moment method for the study of graph spectrum is based on the counting of closed walks due to the following identity: 
$$\int_\mathbb{R} x^r d\mu(G)=q^{-r/2}\cdot  \text{average number of closed walks of length $r$ per vertex}.$$
The value on the right hand side of the above identity, however, can be extremely large and difficult to compute even for graphs of large girth. As noticed by previous works, 
in many cases it is more convienient  to compute the following quantity instead,
$$\int_\mathbb{R} X_{r,q}(x) d\mu(G)=q^{-r/2}\cdot \text{average number of closed NBW of length $r$ per vertex},$$
where $X_{r,q}$ is a polynomial of degree $r$
modified from Chebyshev polynomials.
Roughly speaking, a non-backtracking walk (NBW) is a special type of walk on a graph where, after traversing an edge, the walk does not immediately return along the same edge, see Definition \ref{NBW}. Since there is no non-trivial closed non-backtracking walks of length less than the girth $g(G)$, 
we have for all $1\leq r< g(G)$ that
$$\int_\mathbb{R} X_{r,q}(x) d\mu(G)=0.$$
In particular, the Kesten-McKay distribution $\mu_q$, i.e., the normalized spectral measure of a $(q+1)\text{-}$regular tree, satisfies
$$\int_\mathbb{R} X_{r,q}(x) d\mu_q=0,\ \,\text{for any}\,\, r\geq 1.$$
We mention that $X_{0,q}\equiv 1$. The above identity tells that the polynomials $X_{r,q}, r\geq 1$ are orthogonal to $X_{0,q}$ with respect to the Kesten-McKay distribution $\mu_q$. Indeed, the polynomials $\{X_{r,q}\}_{r=0}^\infty$ form a complete orthogonal basis of $L^2(\mathbb{R},\mu_q)$, see \cite[Equation (24)]{Serre}, \cite[Lemma 2.5]{Sodin} and references therein.

In this paper, we use these Chebyshev polynomials to prove results on spectral measure convergence of (unitary-colored) regular graphs with fixed or growing vertex degrees. In Section \ref{section:4}, by generalizing Friedman \cite[Lemma 3.3]{Friedman91}, recalled in Lemma \ref{ArandXrq} below, to unitary-colored case, we give a short proof of the weak convergence to the Kesten-McKay distribution for the normalized spectral measures of random $N$-lifts. This relies on a criterion as noticed by Sodin \cite[Lemma 2.8]{Sodin}. 
By generalizing this criterion to random graphs of growing degree, in Section \ref{section:5} we study sequences of regular graphs with growing vertex degrees and the convergence of the spectral measures to the semicircle distribution. Dumitriu and Pal \cite{Dumitriu-Pal} show that if $q_n=n^{o(1)}$ and $q_n$ tends to infinity, then the normalized spectral measure of a random $(q_n+1)$-regular graph on $n$ vertices converges weakly to the semicircle distribution in probability. Later, Tran, Vu, and Wang \cite{Tran-Vu-Wang} extends their result to the case that $q_n \leq n/2$. By applying our generalization of Sodin's criterion, we show that the condition that $q_n=n^{o(1)}$ and $q_n$ tends to infinity, in fact guarantees a stronger convergence, that is, convergence \emph{in $p$-Wasserstein distance $W_{p}$ for any $p\in [1,\infty)$} to the semicircle distribution \emph{almost surely} (see Theorem \ref{Dumitriu-Pal}). One advantage of Wasserstein distance between probability measures on $\mathbb{R}$ lies in its explicit expression via the generalized inverse distribution functions (IDF). This approach via Wasserstein distance is motivated by \cite{GHL,GJLS}.

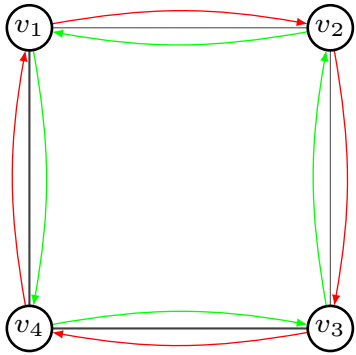
\begin{figure}\label{figure:NBWintroduction}
    \centering
    \begin{tikzpicture}
\Vertex[color=white, x=0, y=4, label=$v_1$, fontscale=1.5]{A};
\Vertex[color=white, x=4, y=4, label=$v_2$, fontscale=1.5]{B};
\Vertex[color=white, x=0, y=0, label=$v_4$, fontscale=1.5]{D};
\Vertex[color=white, x=4, y=0, label=$v_3$, fontscale=1.5]{C};
\Edge[lw=0.1](A)(B);
\Edge[lw=0.8](A)(D);
\Edge[lw=0.1](B)(C);
\Edge[lw=0.8](C)(D);

\Edge[lw=0.5pt, bend=10, color=red, Direct](A)(B);
\Edge[lw=0.5pt, bend=10, color=red, Direct](B)(C);
\Edge[lw=0.5pt, bend=10, color=red, Direct](C)(D);
\Edge[lw=0.5pt, bend=10, color=red, Direct](D)(A);

\Edge[lw=0.5pt, bend=10, color=green, Direct](A)(D);
\Edge[lw=0.5pt, bend=10, color=green, Direct](D)(C);
\Edge[lw=0.5pt, bend=10, color=green, Direct](C)(B);
\Edge[lw=0.5pt, bend=10, color=green, Direct](B)(A);
\end{tikzpicture}
    \caption{The only two closed NBW started from $v_1$ of length 4 in $C_4$.}
    \label{fig:enter-label}
\end{figure}

\section{Preliminaries} 
\subsection{Non-Backtracking Walk on Graphs}
Let $G=(V,E)$ be an undirected graph, where the edge set $E$ may contain multi-edges and multi-loops.  Let $\vec{E}$ be the set of directed edges obtained from $E$, such that each element of $E$ corresponds to two distinct directed edges in $\vec{E}$. For every directed edge $e \in \vec{E}$,  denote the origin (resp., terminus) of $e$ as $o(e)$ (resp., $t(e)$)  and denote its inverse as $\overline{e}$. \\

Recall a \emph{walk} from $a$ to $b$ is a sequence of vertices and edges $$\gamma=(\{v_i\}_{i=0}^{n}, \{e_i\}_{i=0}^{n-1})$$ such that $v_i=o(e_i), v_{i+1}=t(e_i)$, and $v_0=a$, $v_n=b$. The number $n$ is called the \emph{length} of the walk. The set of edges $\{e_i\}_{i=0}^{n-1}$ is allowed to be empty for a walk of length $0$. If $a=b$, we call $\gamma$ a \emph{closed walk}. A walk of length $0$ is seen as a trivial closed walk. For two walks $\gamma_1=(\{v^1_i\}_{i=0}^{n}, \{e^1_i\}_{i=0}^{n-1})$ and $\gamma_2=(\{v^2_j\}_{j=0}^{m}, \{e^2_i\}_{j=0}^{m-1})$ such that $v^1_n=v^2_0$, we define their multiplication to be a new walk $\gamma_1 \gamma_2:=(\{v^1_0,v^1_1,...,v^1_n,v^2_1,v^2_2,...v^2_m\}, 
\{e^1_0,e^1_1,...,e^1_n, e^2_0,e^2_1,...e^2_m\})$ of length $m+n$.\\

\begin{definition}\label{NBW}
    A \emph{non-backtracking walk} is a walk $\gamma=(\{v_i\}_{i=0}^{n}, \{e_i\}_{i=0}^{n-1})$ such that $e_{i}\neq \overline{e_{i+1}}$, $0 \leq i\leq n-2$. We consider the trivial walks non-backtracking. The \emph{girth} of a graph is the length of its shortest nontrivial closed non-backtracking walk. A \emph{circuit} is a non-trivial closed non-backtracking walk $\gamma=(\{v_i\}_{i=0}^{n}, \{e_i\}_{i=0}^{n-1})$ such that $e_{0}\neq \overline{e_{n-1}}$. A circuit $\gamma$ is called \emph{prime} if  $\gamma=\gamma_0^k$ for another circuit $\gamma_0$, then $\gamma=\gamma_0$ and $k=1$.
\end{definition}

\begin{remark}
    As a comparison, a \emph{circle} in a graph $G$ is a (finite) connected subgraph of degree $2$. 
\end{remark}

As an example, in the complete graph $K_4$ illustrated in Figure \ref{2}, the subgraph generated by $\{v_1,v_2,v_3\}$ is a circle, the walk $v_4 \rightarrow v_3 \rightarrow v_2 \rightarrow v_1\rightarrow v_3 \rightarrow v_4$ depicted on the left of Figure \ref{2} is a closed non-backtracking walk but not a circuit, and the walk $v_1 \rightarrow v_2 \rightarrow v_3 \rightarrow v_1 \rightarrow v_4 \rightarrow v_3\rightarrow v_1$ depicted on the right of Figure \ref{2} is a prime circuit.
\ \\
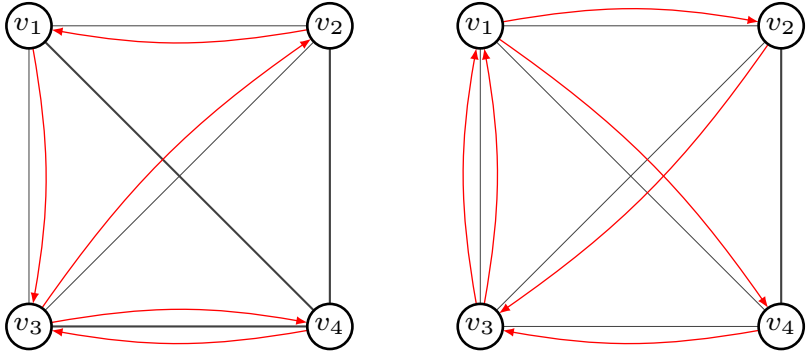
\begin{figure}
    \centering
   \begin{tikzpicture}
\Vertex[color=white, x=0, y=4, label=$v_1$, fontscale=1.5]{A};
\Vertex[color=white, x=4, y=4, label=$v_2$, fontscale=1.5]{B};
\Vertex[color=white, x=0, y=0, label=$v_3$, fontscale=1.5]{C};
\Vertex[color=white, x=4, y=0, label=$v_4$, fontscale=1.5]{D};
\Edge[lw=0.1](A)(B);
\Edge[lw=0.1](A)(C);
\Edge[lw=0.8](A)(D);
\Edge[lw=0.1](B)(C);
\Edge[lw=0.8](B)(D);
\Edge[lw=0.8](C)(D);
\Edge[lw=0.5pt, bend=10, color=red, Direct](C)(D);
\Edge[lw=0.5pt, bend=10, color=red, Direct](D)(C);
\Edge[lw=0.5pt, bend=10, color=red, Direct](C)(B);
\Edge[lw=0.5pt, bend=10, color=red, Direct](B)(A);
\Edge[lw=0.5pt, bend=10, color=red, Direct](A)(C);

\Vertex[color=white, x=6, y=4, label=$v_1$, fontscale=1.5]{E};
\Vertex[color=white, x=10, y=4, label=$v_2$, fontscale=1.5]{F};
\Vertex[color=white, x=6, y=0, label=$v_3$, fontscale=1.5]{G};
\Vertex[color=white, x=10, y=0, label=$v_4$, fontscale=1.5]{H};
\Edge[lw=0.1](E)(F);
\Edge[lw=0.1](E)(G);
\Edge[lw=0.1](E)(H);
\Edge[lw=0.1](F)(G);
\Edge[lw=0.8](F)(H);
\Edge[lw=0.1](G)(H);
\Edge[lw=0.5pt, bend=10, color=red, Direct](E)(F);
\Edge[lw=0.5pt, bend=10, color=red, Direct](F)(G);
\Edge[lw=0.5pt, bend=10, color=red, Direct](G)(E);
\Edge[lw=0.5pt, bend=10, color=red, Direct](E)(H);
\Edge[lw=0.5pt, bend=10, color=red, Direct](H)(G);
\Edge[lw=0.5pt, bend=-10, color=red, Direct](G)(E);
\end{tikzpicture}
    \caption{Examples of closed NBW and circuit}
    \label{2}
\end{figure}

\begin{Notation}
    For a graph $G=(V,E)$, denote the number of closed non-backtracking walks of length $r$ as $f_r(G)$, the number of circuits of length $r$ as $c_r(G)$ and the number of circles of size $r$ as $Z_r(G)$.
\end{Notation}
Notice that by definition, for any graph $G=(V,E)$, there holds
$$c_0(G)\equiv 0\,\,\,\,\,\text{and}\,\,\,\,\,f_0(G)=|V|.$$
On a regular graph, these numbers are related to each other as follows.

\begin{lemma}\label{NBWandcircle}
Let $G=(V,E)$ be a $(q+1)\text{-}$regular graph. The following holds for $r\geq 1$:\\
\begin{equation}\label{frcr}
    f_r=c_r+(q-1)\sum_{1\leq i<  r/2}q^{i-1}c_{r-2i},
\end{equation}
and
\begin{equation}\label{frzr}
    2rZ_r(G)\leq c_r(G) \leq f_r(G)\leq (q+1)^{2}q^{2r-2}\sum_{1\leq k \leq r}kZ_k(G).
\end{equation}
\end{lemma}
For convenience, we provide a quick proof in Appendix \ref{A}. The above identity \eqref{frcr} can be found in \cite[Equation (107)]{Serre}.

\subsection{Non-backtracking Matrices and Chebyshev-type Polynomials}\label{subsection:2.2}
\begin{definition}
Let $G=(V,E)$ be a $(q+1)$-regular graph with adjacency matrix (operator when $G$ infinite) $A(G)$. The \emph{non-backtracking matrix of length r}, denoted as $A_r(G)$, is a matrix (resp. operator) indexed by $V \times V$, such that for $a,b \in V$, $(A_r(G))_{ab}$ equals the number of non-backtracking walks of length $r$ from $a$ to $b$ in $G$. We denote $A=A(G)$, $A_r=A_{r}(G)$, and $||\cdot||$ as the operator norm of the matrix or bounded linear operator if no confusion arises.
\end{definition}

\begin{remark}
    By definition $A_0\equiv I$, $A_1=A$, and $A_2=A^2-(q+1)I$. Here $I$ is the identity matrix (resp. operator).  
\end{remark}
\begin{remark}
    As a comparison, the entry $A^r_{ab}$ of power of the adjacency matrix $A^r$ equals the number of walks length $r$ from $a$ to $b$ in $G$.
\end{remark}

Friedman \cite[Lemma 3.3]{Friedman91} shows that the non-backtracking matrices are related to the adjacency matrix by a series of Chebyshev-type polynomials. We recall the proof here as a warm-up of our proof of its unitary-color generalization Lemma \ref{lemma:color}.

\begin{lemma}\label{ArandXrq} 
For a $(q+1)$-regular graph $G$ where $q\in \mathbb{N}^+$, we have
\begin{equation}\label{Serrechebyshevmoment}
    A_r=q^{r/2}X_{r,q}(q^{-1/2}A),\ r\in \mathbb{N},
\end{equation}
where $X_{r,q},\,r\in \mathbb{N}$ are polynomials defined by 
\begin{equation}\label{Xrqgenerating}
    \sum_{r=0}^{\infty}X_{r,q}(x)t^{r}=\frac{1-q^{-1}t^2}{1-xt+t^2},
\end{equation}
for $|t|$ small enough.
\end{lemma}

\begin{proof}
Notice that for $r\geq 3$, $r\in \mathbb{N}^{+}$, we have
$$A_{r}+qA_{r-2}=A_{r-1}A=A_{r-1}A,$$
and 
$$||A_r||\leq (q+1)q^{r-1}.$$
Thus, it holds that
\begin{equation}\label{recurrencerelationnbmatrices}
    (I-At+qt^2I)\sum_{r=0}^{\infty}A_{r}t^r=(1-t^2)I,
\end{equation}
for $|t|$ small enough. 
On the other hand, we derive from (\ref{Xrqgenerating}) by taking $x=q^{-1/2}A$ that for $|t|$ small enough, 
$$(I-q^{-1/2}At+t^2I)\sum_{r=0}^{\infty}X_{r,q}(q^{-1/2}A)t^{r}=(1-q^{-1}t^2)I,$$
or equivalently
\begin{equation}\label{nbcheby}
    (I-At+qt^2I)\sum_{r=0}^{\infty}q^{r/2}X_{r,q}(q^{-1/2}A)t^{r}=(1-t^2)I.
\end{equation}
Comparing (\ref{recurrencerelationnbmatrices}) and (\ref{nbcheby}) yields the result.
\end{proof}
The explicit expressions of the polynomials $X_{r,q}$ are given below. The proofs are provided in Appendix \ref{B}. 
\begin{lemma}\label{Xrqexplicit}
    For $q\in \mathbb{N}^+$ and $r\in \mathbb{N}$, the polynomial $X_{r,q}$ defined by (\ref{Xrqgenerating}) is of degree $r$ given by
    \begin{equation}\label{eq:Xrq}
        X_{r,q}=X_r-q^{-1}X_{r-2}, 
    \end{equation}
    where 
    \begin{equation}\label{chebyshevexplicit}
    X_r(x):=\sum_{0\leq k \leq r/2}(-1)^k\binom{r-k}{k}x^{r-2k},
\end{equation}
and we use the convention that $X_{r}\equiv 0$ for $r \in \mathbb{Z}^{-}$.
\end{lemma}
\begin{remark}
 Let us denote $Y_r:=X_{r,1}$ following Serre \cite{Serre}. For $r\geq 1$, $Y_r$ and $X_r$ are related to the \emph{Chebyshev polynomials of the first and second kind} $T_r$ and $U_r$, respectively, by a change of variables. Indeed, we have
   $$Y_r(x)=2T_r(x/2),$$
   $$X_r(x)=U_r(x/2).$$
The polynomials $Y_r$ and $X_r$ are also referred as \emph{Vieta–Lucas polynomials and Vieta–Fibonacci polynomials}, respectively, in \cite{Horadam}. 

\begin{comment}
    We mention that the Chebyshev polynomials $U_r$ is defined via 
\begin{equation}\label{eq:Ur}
    U_r(\cos\theta)=\frac{\sin (r+1)\theta}{\sin \theta}.
\end{equation}
The explicit expression is given below:
$$U_r(x)=\sum_{0\leq k \leq r/2}(-1)^k\binom{r-k}{k}(2x)^{r-2k}.$$
\end{comment}

\end{remark}

As a corollary, the Chebyshev-type polynomials are related to the number of non-backtracking walks as shown in \cite[Euqations (108)-(110)]{Serre}, see also \cite{Friedman91}.

\begin{corollary}\label{cor:trace}
    Let $G=(V,E)$ be a $(q+1)\text{-}$regular graph, then for $r \in \mathbb{N}$,
    \begin{equation}\label{1stchebgeo}
    \mathrm{Tr}(X_{r,q}(q^{-1/2}A))=q^{-r/2}f_r,
\end{equation}
\begin{equation}\label{2ndchebgeo}
\mathrm{Tr}(X_{r}(q^{-1/2}A))=q^{-r/2}\sum_{0\leq k\leq r/2}f_{r-2k},
\end{equation}
\begin{equation}\label{circuitnumber}
    \mathrm{Tr}(Y_{r}(q^{-1/2}A))= q^{-r/2}c_r-\left\{
\begin{aligned}
    &0,\ r=2k-1, k\in \mathbb{N}^+;\\
    &(q-1)q^{-r/2}|V|,\ r=2k, k\in \mathbb{N}^+.
\end{aligned} \right.
\end{equation}
\end{corollary}

\begin{comment}
    \begin{proof}
We first observe from (\ref{eq:Xrq}) that there holds
\begin{equation}\label{chebyinversehao}
    X_r=\sum_{0\leq k\leq r/2}q^{-k}X_{r-2k,q},
\end{equation}
and recall $Y_r=X_{r,1}$.
Then, 
    (\ref{1stchebgeo}) is due to (\ref{Serrechebyshevmoment}), and (\ref{2ndchebgeo}) comes from (\ref{1stchebgeo}) and (\ref{chebyinversehao}). Combining (\ref{2ndchebgeo}) with (\ref{frcr}) we get (\ref{circuitnumber}).
\end{proof}
\end{comment}

\subsection{Normalized spectral measures of graphs}

We are interested in the following probability measures related to graphs.
\begin{definition}[Spectral measure of regular finite graph]\label{spectralmeasure}
    Let $G=(V,E)$ be a finite $(q+1)$-regular graph with adjacency matrix $A$. Denote the $n$ eigenvalues of $A$ as $\lambda_{|V|}(A)  \leq \cdots \leq \lambda_2(A) \leq \lambda_1(A)$, the normalized spectral measure of $G$ is the following probability measure of bounded supported on $\mathbb{R}$
\begin{equation}
    \mu(G):=\frac{1}{|V|}\sum_{1\leq k\leq |V|}\delta_{q^{-1/2} \lambda_k(A)}.
\end{equation}
Equivalently, $\mu_G$ is defined such that for any polynomial $P\in \mathbb{R}[x]$, 
\begin{equation}
    \int_\mathbb{R} P d\mu= \frac{1}{|V|}\mathrm{Tr}( P(q^{-1/2}A)).
\end{equation}
\end{definition}
\begin{remark}
    In a probabilistic context, the spectral measure defined above is often referred as the \emph{empirical spectral distribution} (\emph{ESD}) of $q^{-1/2} A$.
\end{remark}

The spectral measure can also be defined for infinite graphs which are vertex-transitive. 
We first observe for a finite vertex-transitive graph $G$, and a polynomial $P\in \mathbb{R}[x]$ that
$$\frac{1}{|V|}\mathrm{Tr}( P(q^{-1/2}A))=\langle \mathbf{1}_{v}, P(A)  \mathbf{1}_{v} \rangle, \ \ \text{for any}\,\, v \in V_G,$$ 
since  $$\langle \mathbf{1}_{v_1}, P(A)  \mathbf{1}_{v_1} \rangle =\langle \mathbf{1}_{v_2}, P(A)  \mathbf{1}_{v_2} \rangle, \ \ \text{for any}\,\, v_1,v_2 \in V_G. $$
We can then define the spectral measure of any vertex-transitive graph (not necessarily finite) as below.

\begin{definition}[Spectral measure of vertex-transitive graphs]\label{defn:spectralmeasureinfinite}
Let $o$ be any given vertex on a vertex-transitive graph $G$. Suppose that $G$ is $(q+1)$-regular. Then the normalized spectral measure $\mu(G)$ is the unique probability measure of bounded support  on $\mathbb{R}$ such that for any polynomial $P$, $$\int_\mathbb{R} Pd\mu(G)=\langle \mathbf{1}_o, P(q^{-1/2}A)\mathbf{1}_o\rangle.$$
\end{definition}

\begin{remark}
    The existence and uniqueness of the spectral measure of an infinite vertex-transitive graph is due to the adjointness of the normalized adjacency operator $q^{-1/2}A$, see Hall \cite[Proposition 7.17]{Hall}. 
\end{remark}

Notice that the $(q+1)\text{-}$regular tree $\mathbb{T}_{q}$ is vertex transitive. Indeed, it can be viewed as the Cayley graph of $*_{q+1} \mathbb{Z}_2$.
\begin{definition}
    The normalized spectral measure $\mu_q=\mu(\mathbb{T}_{q})$ of the $(q+1)\text{-}$regular tree is referred as \emph{Kesten-McKay distribution} or \emph{Kesten-McKay law}. It satisfies:
    \begin{equation}\label{defofkm}
    \int_{\mathbb{R}}Pd\mu_{q}=\langle\mathbf{1}_o, P(q^{-1/2}A(\mathbb{T}_q))\mathbf{1}_o\rangle, \ \text{for any polynomial} \ P.
    \end{equation}
\end{definition}

\subsection{Convergence of measures in the Wasserstein Space}

In this section, we review various types of convergence of probability measures.

\begin{definition}
    Define $P(\mathbb{R})$ as the collection of all probability measures on $\mathbb{R}$ and $C_b(\mathbb{R})$ as the collection of all bounded continuous function on $\mathbb{R}$. We say $\mu_n$ \emph{converges weakly}  to $\mu$ in $P(\mathbb{R})$, denoted as 
$\mu_n \rightharpoonup \mu$, if for any $f \in C_b(\mathbb{R})$, 
\begin{equation}
    \lim_{n \rightarrow \infty }\int_\mathbb{R} f d\mu_n= \int_\mathbb{R} f d\mu.
\end{equation}
\end{definition}

The weak topology is in fact metrizable, see, for example, in Villani \cite[Corollary 6.13]{Villani}. 
\begin{lemma}\label{lemma:weakconvergence}
    There exists a metric on $P(\mathbb{R})$, such that $\mu_n \rightarrow \mu$ in this metric space if and only if $\mu_n \rightharpoonup \mu.$
\end{lemma}
\begin{remark}\label{rmk:weakconvergence}
    Such a metric is not unique. We denote any such metric as $d_W$.
\end{remark}
The spectral measures of regular graphs defined in Definitions \ref{spectralmeasure} and \ref{defn:spectralmeasureinfinite} always have bounded supports, thus have finite $p$-moments for any $p \in [1,\infty)$. Motivated by this fact, we study these measures in the following \emph{Wasserstein space}.

\begin{definition}[The Wasserstein space $P_p(\mathbb{R})$]
    For any $p \in [1,\infty)$, define $P_p(\mathbb{R})$ as the collection of all probability measures on $\mathbb{R}$ with finite $p$ moment, i.e.,
    $$P_p(\mathbb{R})=\{\mu \in P(\mathbb{R}): \int_\mathbb{R} |x|^p d\mu(x) <\infty\}.$$ For $p=\infty$, define $P_{\infty}(\mathbb{R})$ to be the collection of all probability measures with bounded support.
\end{definition}
The Wasserstein space $P_p(\mathbb{R})$ is equipped with the following metric $W_p$.
\begin{definition}[The Wasserstein distance]
For any $\mu, \nu \in P_p(\mathbb{R})$, $p \in [1,\infty)$, define the Wasserstein distance as
\begin{equation}
     W_p(\mu,\nu)=\inf_{\gamma \in \Pi (\mu,\nu)} \left(\int_{\mathbb{R}^2} |x-y|^p d\gamma(x,y)\right)^{1/p},
\end{equation}
and 
\begin{equation}
     W_{\infty}(\mu,\nu)=\inf_{\gamma \in \Pi (\mu,\nu)} ||x-y||_{L^{\infty}(\mathbb{R}^2, \gamma)},
\end{equation}
where $\Pi (\mu,\nu)$ denotes the collection of all couplings between $\mu$ and $\nu$.
\end{definition}

We have the following more explicit formula for Wasserstein distance, see, e.g., \cite[Proposition 2.17]{Santambrogio}.
\begin{lemma}\label{closedformwasserstein}
For any $\mu, \nu \in P_p(\mathbb{R})$, $p \in [1,\infty]$,
    \begin{equation}
    W_p(\mu,\nu)=||F^{-1}_{\mu}-F^{-1}_{\nu}||_{L^p[0,1]},
\end{equation}
where $F^{-1}_{\mu}$ is the \emph{generalized inverse distribution function (IDF)}, i.e., 
$$F^{-1}_{\mu}(x)=\inf \{y\ |\ \mu((-\infty, y]) \geq x\}.$$
\end{lemma}

\begin{remark}
    Notice that $W_p(\mu,\nu) \leq W_q(\mu,\nu)$ if $p\leq q$ and 
    $$\lim_{p\rightarrow \infty} W_{p}(\mu,\nu)=W_{\infty}(\mu,\nu).$$
\end{remark}

\begin{remark} Let $\mu_G$ be the normalized spectral measure of a finite graph $G$. Then
 the generalized inverse distribution function (IDF) of $\mu_G$ is actually a simple function on $[0,1]$, such that $$F^{-1}_{\mu_G}(x)=q^{-1/2} \lambda_{|V|-k}(A),\ \ x \in \left(\frac{k}{m}, \frac{k+1}{m}\right].$$
\end{remark}

For $p\in [1,\infty)$, the Wasserstein space $P_p(\mathbb{R})$ has the following property, see, e.g., \cite[Definition 6.8, Theorem 6.9 and Theorem 6.18]{Villani}.
\begin{lemma}\label{criterionwasserstein}
    For $p \in [1,\infty)$, the metric space $(P_p(\mathbb{R}), W_p)$ is complete and separable. Moreover, $\mu_n \rightarrow \mu$ in $P_p(\mathbb{R})$ if and only if $\mu_n \rightharpoonup \mu$ and 
    $$\int_\mathbb{R} |x|^p d\mu_n(x) \rightarrow \int_\mathbb{R} |x|^p d\mu(x).$$
\end{lemma}

\begin{remark}
    Lemma \ref{criterionwasserstein} does not hold for $p=\infty$.
\end{remark}
\begin{definition}
    Suppose that $\mu$ and $\mu_n$ are probability measures on $\mathbb{R}$ with finite $p$ moment for any $p\in [1,\infty)$. We say $\mu_n$ \emph{converges in moments} to $\mu$, if for any $P \in \mathbb{R}[x]$, $$\int_\mathbb{R} Pd\mu_n \rightarrow \int_\mathbb{R} Pd\mu, \,\,\text{as}\,\,n\to \infty.$$
\end{definition}

The moment method is built upon the following key fact, see, e.g.,  Fleermann-Kirsch \cite[Theorem 3.5]{Fleermann-Kirsch}.
\begin{lemma}[Convergence in moments implies weak convergence]\label{momentimpliesweak}
    Suppose that $\mu$ and $\mu_n$ are probability measures on $\mathbb{R}$ with finite $p$ moment for any $p\in [1,\infty)$. If $\mu_n$ converges in moments to $\mu$ and $\mu$ is uniquely determined by its moments, then $\mu_n$ converges weakly to $\mu$.
\end{lemma}

\begin{remark}
    Any probability measure of bounded support is uniquely determined by its moments, see Fleermann-Kirsch \cite[Corollary 3.4]{Fleermann-Kirsch}. Thus, in order to prove weak convergence of $\mu_n$ to a probability measure $\mu$ of bounded support, it suffices to prove the convergence in moments. Moreover, if the support of $\mu_n$ is uniformly bounded, by choosing a suitable cut-off function, one can see that $\mu_n$ converges weakly to $\mu$ is equivalent to $\mu_n$ converges in moments to $\mu$.
\end{remark}

\begin{lemma}\label{equivalentconvergence}
    Suppose that $\mu$ and $\mu_n$ are probability measures on $\mathbb{R}$ with finite $p$ moment for any $p\in [1,\infty)$, and $\mu$ is uniquely determined by its moments. Then $\mu_n$ converges in moments to $\mu$ if and only if it holds for any $p\in [1,\infty)$ that $W_p(\mu_n, \mu) \rightarrow 0\ \text{as}\ n\rightarrow \infty$. 
\end{lemma}
\begin{proof}
    This follows from Lemma \ref{criterionwasserstein} and Lemma \ref{momentimpliesweak}.
\end{proof}

For a sequence of random regular graphs, each picked from a series of graph ensembles independently, we regard the spectral measures associated as a sequence of independent random variables taking values in the metric space $(P(\mathbb{R}), d_W)$ or $(P_p(\mathbb{R}), W_p)$.

\begin{definition}\label{defn:convergence}
    Let $G_n=(V_n, E_n)$ be a sequence of independent random regular graphs with normalized spectral measure $\mu_n$. Let $\mu$ be a deterministic probability measure. We define the following three convergence:
    \begin{itemize}
        \item [(i)] $\mu_n$ \emph{converges weakly to $\mu$  a.s.} if $$d_W(\mu_n, \mu) \rightarrow 0\ \text{as}\ n\rightarrow \infty\ a.s..$$
\item[(ii)] $\mu_n$ \emph{converges weakly to $\mu$  in probability} if for any $\varepsilon>0$, $$\mathbf{P}(d_W(\mu_n, \mu)>\varepsilon) \rightarrow 0\ \text{as}\ n\rightarrow \infty.$$
\item[(iii)] Suppose $\mu$ be uniquely determined by its moments. $\mu_n$ \emph{converges in moments to $\mu$  a.s.} if for any $p\in [1,\infty)$,  $$W_p(\mu_n, \mu) \rightarrow 0\ \text{as}\ n\rightarrow \infty\ a.s..$$
\end{itemize}
\end{definition}

The Definition \ref{defn:convergence} (i) is independent of the choices of the metric $d_W$, due to Lemma \ref{lemma:weakconvergence} and Remark \ref{rmk:weakconvergence}. 
 In fact, so is the Definition \ref{defn:convergence} (ii), due to the following lemma, see, e.g., Fleermann-Kirsch \cite[Theorem 2.25]{Fleermann-Kirsch}.
\begin{lemma}
Let $G_n=(V_n, E_n)$ be a sequence of independent random regular graphs. The corresponding normalized spectral measure $\mu_n$ converges weakly to $\mu$  in probability if and only if for any 
$f_\mathbb{R} \in C_b(\mathbb{R})$, the real-valued random variable $\int_\mathbb{R} f d\mu_n$ converges in probability to $\int f d\mu$.
\end{lemma}

\subsection{Kesten-McKay Distribution and Orthogonal Relations}\label{section:3}

The Kesten-McKay distribution is closely related to the following semicircle distribution.
\begin{definition}
     The \emph{Wigner semicircle distribution} (also known as  semicircle law) $\mu_{\infty}$ is defined to be 
$$d\mu_{\infty}(x)=\frac{1}{2\pi}\sqrt{4-x^2}\mathbf{1}_{|x|\leq 2}\ dx.$$
\end{definition}
In the number theoretic literature, the Wigner semicircle distribution is also known as Sato-Tate distribution.

The following Theorem \ref{thm:McKayLaw} gives the explicit formula of the Kesten-McKay distribution and also the orthogonal relations of the Chebyshev polynomials, see \cite[Section 2.2]{Serre}, \cite[Lemma 2.5]{Sodin} and references therein. 
We provide a proof in Appendix \ref{C}.
\begin{theorem}\label{thm:McKayLaw}
    The Kesten-McKay distribution $\mu_{q}$ for $q=1$ is given by 
    \begin{equation}\label{mu1explicit}
d\mu_{1}(x)=\frac{1}{\pi}\frac{1}{\sqrt{4-x^2}}\mathbf{1}_{|x|\leq 2}\ dx,
\end{equation}
and for $q\geq 2$,
    \begin{equation}\label{muqexplicit}
    d\mu_{q}(x)=\frac{1}{2\pi}\frac{(q+1)\sqrt{4-x^2}}{ (q^{-1/2}+q^{1/2})^2-x^2}\mathbf{1}_{|x|\leq 2}\ dx.
\end{equation} 
The polynomials $\{X_{r,q}\}_{r=0}^{\infty}$ defined in (\ref{eq:Xrq}) 
satisfy 
\begin{equation}\label{orthogonalrelation}
\int_\mathbb{R} X_{n,q} X_{m,q}d\mu_{q}=\left\{
\begin{aligned}
    &0, \ m\neq n;\\
    &1, \ m=n=0;\\
    &1+q^{-1}, \ m=n\neq 0.
\end{aligned} \right.
\end{equation}
and hence form a complete orthogonal basis of $L^2(\mathbb{R},\mu_{q})$.
\end{theorem}
Figure \ref{figure:KestenMcKay} illustrates the density functions of Kesten-McKay distribution $\mu_1$, $\mu_3$ and $\mu_{20}$.

\begin{remark}
   After a shift of variable, the distribution $\mu_1$ becomes the arcsine law with density function
    \begin{equation}
\rho(x)=\frac{1}{\pi}\frac{1}{\sqrt{x(4-x)}}\mathbf{1}_{0\leq x\leq 4}.
\end{equation}
\end{remark}

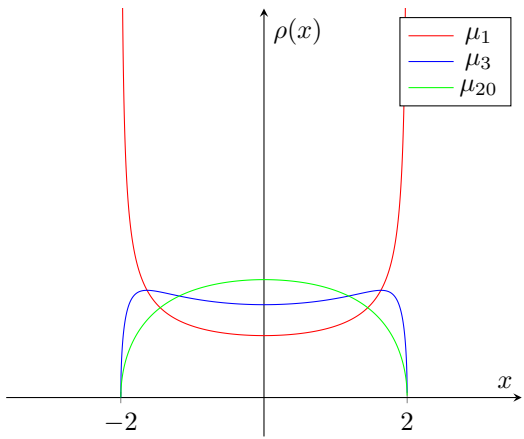
\begin{figure}
    \centering
   \begin{tikzpicture}
\begin{axis}[xmin=-3.6,xmax=3.6,ymin=-0.1,ymax=1,
    axis lines = middle,
    ytick={0},
    yticklabels={},
    xtick={-2, 0, 2},
    xlabel = \(x\),
    ylabel = {\(\rho(x)\)},
]

\addplot[
    domain=-1.9999:1.9999,
    samples=1000,
    color=red]{(4-x^2)^(-0.5)/(pi)};
\addlegendentry{\(\mu_{1}\)};
\addplot[
    domain=-2:2,
    samples=1000,
    color=blue]{12*(4-x^2)^(0.5)/(2*pi*(16-3*x*x))};
\addlegendentry{\(\mu_{3}\)};
\addplot[
    domain=-2:2,
    samples=1000,
    color=green]{20*21*(4-x^2)^(0.5)/(2*pi*(21*21-20*x*x))};
\addlegendentry{\(\mu_{20}\)}

\end{axis}
\end{tikzpicture}
    \caption{Density function of the Kesten-McKay distribution when $q=1,3,20$}
    \label{figure:KestenMcKay}
\end{figure}

\begin{remark}\label{betaJacobi}
    One can check that for all $q>1$, $q\in \mathbb{R}$, $\mu_q$ defined as in (\ref{muqexplicit}) is a probability measure on $\mathbb{R}$, and $\{X_{r,q}\}_{r=0}^{\infty}$ defined as in (\ref{eq:Xrq}) form a complete orthogonal basis of $L^2(\mathbb{R},\mu_{q})$. Notice that $\mu_{q}$ converges weakly to $\mu_{1}$ and $\mu_{\infty}$ as $q$ tends to $1^{+}$ and $\infty$. For fixed $r$, $X_{r,q}$ tends to the polynomials $Y_r$ and $X_r$ as $q$ tends to  $1^{+}$ and $\infty$, respectively. Recall that $Y_r$ and $X_r$ are obtained from the Chebyshev polynomials of the first and second kind via a change of variables.
    %Such probability measures are closely related to the limiting distribution of $\beta\text{-}Jacobi$ ensembles (Compare these with Proposition 1.1 of \cite{Dumitriu-Paquette}).
\end{remark}
\section{Random regular graphs and random lifts}\label{section:4}
Given a $(q+1)$-regular graph $G$. Let $G_n$ be a random $n$-lift of $G$. 
In this section, we show the normalized spectral measures of $G_n$ converge weakly in probability to Kesten-McKay distribution. For that purpose, we need a unitary-colored generalization of Lemma \ref{ArandXrq} and a criterion due to Serre \cite{Serre} and Sodin \cite{Sodin}. 
\subsection{A Criterion}
The following criterion as noticed by \cite[Lemma 2.8]{Sodin} and also \cite[Theorem 10]{Serre} characterize the convergence of spectral measures of random regular graphs via counting non-backtracking walks or circles. We provide a proof in Appendix \ref{D} for convenience.

\begin{theorem}\label{kestenmain}
    For a series of independent $(q+1)$-regular random graphs $G_n=(V_n, E_n)$, the following are equivalent as $n\rightarrow \infty$:
    \begin{itemize}
        \item [(i)] The normalized spectral measure $\mu_{n}=\mu(G_n)$ converges almost surely to the Kesten-McKay distribution $\mu_q$;
    \item[(ii)] For any $r \geq 1$, $$\frac{f_{r}(G_n)}{|V_n|} \rightarrow 0\  \text{almost surely};$$
    \item[(iii)] For any $r \geq 1$, $$\frac{Z_{r}(G_n)}{|V_n|} \rightarrow 0\  \text{almost surely}.$$
    \end{itemize}
    The theorem holds still if every "almost surely" is replaced by "in probability".
\end{theorem}

\begin{comment}
    Let $G(n,q+1)$ be the collection of all $(q+1)\text{-}$regular simple graphs on $n$ vertices with uniform probability measure.  Then we have the following convergence to Kesten-McKay distribution due to the lack of circles. The result is originally due to McKay \cite{McKay}.
\begin{corollary}\label{cor:McKay}
    Let $G_n \in G(n,q+1)$ with adjacency matrix $A_n$, then the normalized spectral measures of $G_n$ converge weakly to the Kesten-McKay distribution $\mu_q$ almost surely.
\end{corollary}
\begin{proof}
    First notice that $Z_1(G_n)=Z_2(G_n)\equiv0$. For $k\geq 3$, it is known that as $n\rightarrow \infty$, the random variable $Z_k(G_n)$ converges in distribution to a Poisson distributed random variable with parameter $\lambda_k:=\frac{1}{2k}q^k$ and $$\mathbb{E}Z_k^2(G_n) \rightarrow \lambda_k^2+\lambda_k,$$ see \cite{Bollobas} and \cite{Wormald}. In this way for any $k\geq 1$,
    $$\mathbb{E} \sum_{n=1}^{\infty} \frac{Z^2_k(G_n)}{n^2} < \infty.$$
    By Borel–Cantelli lemma, $\frac{Z_{k}(G_n)}{n} \rightarrow 0$ a.s. and the corollary holds.
\end{proof}
\end{comment}

\subsection{Colored non-backtracking matrices}

 Let $G=(V,E)$ be an undirected $(q+1)\text{-}$ regular graph with directed edge set $\vec{E}$. Let $\mathcal{G}$ be a subgroup of $U(N)$ with $N\in \mathbb{N}^+$. A $\mathcal{G}$-color on $G$ is an assignment $\sigma: \Vec{E} \rightarrow \mathcal{G}$, such that $\sigma(\overline{e})=\sigma(e)^{-1}$(the symmetric condition). We treat the group elements of $\mathcal{G}$ as matrices of size $N \times N$. Define $A^{\sigma}$ to be the block matrix of size $N|V|\times N|V|$ such that for any $i,j\in V$
$$(A^{\sigma})_{ij}= \sum_{e\in \vec{E}:\, o(e)=i, t(e)=j}\sigma_{e}.$$ 
If there is no such edge $e\in \vec{E}$ with $o(e)=i$ and $t(e)=j$, we use the convention that $(A^\sigma)_{ij}$ equals the null matrix. Note that the matrix $A^\sigma$ is self-adjoint. Such a pair $(G,\sigma)$ has been referred to in the literature as voltage graph, gain graph or connection graph, see, e.g., \cite{KMY} and the references therein.   

For any walk $\gamma=(\{v_i\}_{i=0}^{n},\{e_i\}_{i=1}^{n})$ with $n\geq 1$, we define its $\mathcal{G}$-color via multiplication on the right as $$\sigma_{\gamma}=\prod_{i=1}^{n}\sigma_{e_i}.$$ We now present our definition of $\mathcal{G}$-colored non-backtracking matrices.
\begin{definition} Let $G=(V,E)$ be a graph associated with a $\mathcal{G}$-color $\sigma$, where $\mathcal{G}$ is a subgroup of $U(N)$.
    The \emph{$\mathcal{G}$-colored non-backtracking matrix $A_r^\sigma$ of length $r$} is defined as a block matrix such that $$(A_{r}^{\sigma})_{ij}=\sum_{\gamma \in P_{ij}(r)}\sigma_{\gamma},\,\,\text{for}\,\,r\geq 1,$$ where $P_{ij}(r)$ stands for the collection of all non-backtracking walks from $i$ to $j$ of length $r$ and $A_{0}^{\sigma}=\mathrm{I}$ is the identity matrix of size $N|V|\times N|V|$.
\end{definition}
 Note that each $(A^{\sigma}_r)_{ij}$ is a %self-adjoint 
 polynomial of $N\times N$ matrices of degree $r$ with coefficients depending only on the graph structure of $G$. 
The following lemma, which connects the colored non-backtracking matrices with Chebyshev-type polynomials, is a colored version of Lemma \ref{ArandXrq}.
\begin{lemma}\label{lemma:color} Let $G=(V,E)$ be a graph associated with a $\mathcal{G}$-color $\sigma$, where $\mathcal{G}$ is a subgroup of $U(N)$. For any $q\in \mathbb{N}^+$ and $r\in \mathbb{N}$, we have
    \begin{equation}\label{Serretrace}
    A^{\sigma}_r=q^{r/2}X_{r,q}(q^{-1/2}A^{\sigma}).
\end{equation}
\end{lemma}
\begin{proof}
By definition, the following recurrence relation holds for the colored non-backtracking matrix $A_r^\sigma$:
$$A_1^{\sigma}=A^{\sigma},$$
$$A_2^{\sigma}=(A^{\sigma})^2-(q+1)I,$$
and for $r\geq 3$, $r\in \mathbb{N}^+$,
$$A^{\sigma}_{r}+qA^{\sigma}_{r-2}=A^{\sigma}_{r-1}A^{\sigma}=A^{\sigma}A^{\sigma}_{r-1}.$$
We define a new matrix via the polynomials $X_{r,q}$ as below:
\[\overline{A^{\sigma}_r}:=q^{r/2}X_{r,q}(q^{-1/2}A^{\sigma}).\]
Due to the generating function (\ref{Xrqgenerating}) of $X_{r,q}$, the above recurrence relation holds still for the matrix $\overline{A_r^{\sigma}}$.
Therefore, we have $\overline{A_{r}^{\sigma}}=A_{r}^{\sigma}$. This completes the proof.
\end{proof}

We define the normalized spectral measure of the pair $(G,\sigma)$ as follows.
\begin{definition}\label{spectralmeasurecolored}
    Let $G=(V,E)$ be a finite $(q+1)$-regular graph with a $\mathcal{G}$-color $\sigma$, with $\mathcal{G}$ being a subgroup of $U(N)$. Denote the $N|V|$ eigenvalues of $A^\sigma$ as $\lambda_{N|V|}(A^\sigma) \leq \cdots \leq \lambda_2(A^\sigma) \leq \lambda_1(A^\sigma)$, the normalized spectral measure of $(G,\sigma)$ is the following probability measure on $\mathbb{R}$
\begin{equation}
    \mu(G,\sigma):=\frac{1}{N|V|}\sum_{1\leq k\leq N|V|}\delta_{q^{-1/2} \lambda_k(A^\sigma)}.
\end{equation}
\end{definition}

\begin{theorem}\label{thm:colormckay}
Let $\{G_n\}_{n=0}^\infty$ be a sequence of $(q+1)$-regular graphs. Let $\mathcal{G}_{n}$ be a subgroup of $U(N_n)$ for all $n$. Then the normalized spectral measure $\mu(G_n)$ converges weakly to the Kesten-McKay distribution $\mu_q$ if and only if the measure $\mu(G_n,\sigma_n)$ converges weakly to $\mu_q$ for any $\mathcal{G}_{n}$-colors $\{\sigma_n\}_{n=0}^\infty$.
\end{theorem}
\begin{proof}
We only need to show the only if part.
Let $G_n=(V_n,E_n)$ and $\sigma_n$ be a $\mathcal{G}_{n}$-color with $\mathcal{G}_{n}$ being a subgroup of $U(N_n)$.
    We first observe from Lemma \ref{lemma:color} that
    \begin{equation*}
        \int_\mathbb{R} X_{r,q}d\mu(G_n,\sigma_n)=\frac{q^{-r/2}}{N_n|V_n|}\sum_{i\in V_n}\mathrm{Tr}((A_r^{\sigma_n})_{ii}).
    \end{equation*}
    Since every unitary matrix has norm one, we have $$\sum_{i\in V_n}\left|\mathrm{Tr}((A_r^{\sigma_n})_{ii})\right|\leq \sum_{i\in V_n}N_n\Vert(A_r^{\sigma_n})_{ii}\Vert\leq N_{n}\sum_{i\in V_n}\sum_{\gamma \in P_{ii}(r)}||\sigma(\gamma)||\leq N_nf_r(G_n).$$
    Therefore, we obtain
    \[\left|\int_\mathbb{R} X_{r,q}d\mu(G_n,\sigma_n)\right|\leq q^{-r/2}\frac{f_r(G_n)}{|V_n|}.\]
    By Theorem \ref{kestenmain}, the measure $\mu(G_n)$ converges weakly to $\mu_q$ if and only if $f_r(G_n)/|V_n|$ tends to $0$ as $n$ tends to $\infty$, for any $r\geq 1$. The latter property implies that $\mu(G_n, \sigma_n)$ converges to $\mu_q$ in moments. By Lemma \ref{momentimpliesweak}, this tells that $\mu(G_n, \sigma_n)$ converges weakly to $\mu_q$. 
\end{proof}

\subsection{Random Lifts}

We now prove convergence to the Kesten-McKay distribution of normalized spectral measures of random lifts by lack of closed non-backtracking walks. Part of our proof is motivated by Bordenave-Collins \cite{Bordenave-Collins}. Let $G=(V,E)$ be an undirected $(q+1)\text{-}$ regular graph.
We assign to $G$ a $\mathcal{G}$-color $\sigma$ where $\mathcal{G}=\pi_N$ is 
the permutation group on $N$ elements. Note that each assignment of a $\pi_N$-color on $G$ represents an $N$-lift (covering) of $G$, denoted by $G^{\sigma}$, and the matrix  $A^{\sigma}$ is exactly the adjacency matrix of the lift graph $G^{\sigma}$. 

We consider the following model of random $N$-lifts of the graph $G$, which is the ensemble of all $\{G^\sigma: \,\,\sigma \,\,\text{is a $\pi_N$-color} \}$ with uniform probability measure.

\begin{theorem}\label{McKayrandomlift}
 Let $G$ be a finite $(q+1)$-regular graph and $G_n$ be the random $n$-lift of $G$. Then the normalized spectral measure of $G_n$ converges weakly to the Kesten-McKay distribution $\mu_q$ in probability.
\end{theorem}

The lack of closed non-backtracking walks is guaranteed by the asymptotically freeness of permutation matrices. Next, we employ Lemma \ref{lemma:color} to prove Theorem \ref{McKayrandomlift}. We remark that, by
comparing Lemma \ref{lemma:color} and Lemma \ref{ArandXrq}, the $\pi_N$-colored non-backtracking matrices $A_r^\sigma$ are precisely the non-backtracking matrices of the corresponding $N$-lift graph $G^{\sigma}$.
%We now prove Theorem \ref{McKayrandomlift}.
\begin{proof}[Proof of Theorem \ref{McKayrandomlift}]
    Combining Lemma \ref{lemma:color} and the identity (\ref{1stchebgeo}) in Corollary \ref{cor:trace} leads to
    $$q^{-r/2}\frac{f_r(G^{\sigma})}{|V(G^{\sigma})|}=\frac{1}{n|V|} \mathrm{Tr}(A^{\sigma}_r)=\frac{1}{|V|}\sum_{i \in V} \frac{1}{n}\mathrm{Tr}((A^{\sigma}_r)_{ii}).$$ Notice that $(A^{\sigma}_r)_{ii}$ is a polynomial of random permutation matrices. Taking expectations yields
    $$q^{-r/2}\mathbb{E}\frac{f_r(G^{\sigma})}{|V(G^{\sigma})|}=\frac{1}{|V|}\sum_{i \in V} \mathbb{E}\frac{1}{n}\mathrm{Tr}((A^{\sigma}_r)_{ii})=\frac{1}{|V|}\sum_{i \in V}\sum_{\gamma\in P_{ii}(r)}\mathbb{E}\frac{1}{n}\mathrm{Tr}(\sigma_{\gamma}).$$
    We recall the following result of Nica \cite[Equation (7)]{Nica}: Let $t_1,\ldots, t_k$ be uniformly and independently chosen from $\pi_n$. Then it holds that
      \begin{equation}
\lim_{n\rightarrow\infty}\mathbb{E}\frac{1}{n}t_{c_1}^{\alpha_1}t_{c_2}^{\alpha_2}...t_{c_m}^{\alpha_m}\rightarrow 0,
\end{equation}
for any $c_1\neq c_2\neq \cdots\neq c_m$ taking values from $\{1,2,\ldots,k\}$ and any $\alpha_1,\alpha_2,\ldots, \alpha_m$
from $\mathbb{Z}\setminus\{0\}$.
\\
Thus as $n\rightarrow \infty$, we have $$\mathbb{E}\frac{f_r(G^{\sigma})}{|V(G^{\sigma})|}\rightarrow 0.$$
This implies that ${f_r(G^{\sigma})}/{|V(G^{\sigma})|}\rightarrow 0$ in probability. Therefore, Theorem \ref{McKayrandomlift} follows directly from Theorem \ref{kestenmain}.
\end{proof}
\begin{comment}
\begin{remark}
    Notice that for any $d\in \mathbb{N}^+$,  every $2d\text{-}$regular graph is a covering graph of the wedge sum $\vee_d \ \mathbf{S}^1$ of $d$ copies of $\mathbf{S}^1$, viewing as a graph of one vertex and $d$ loops \cite{Gross}.  The \emph{permutation model} of random regular graphs of degree $2d$ and $n$ vertices is then defined as the random $n$ lift of the wedge sum $\vee_d \ \mathbf{S}^1$. In this way Theorem \ref{McKayrandomlift} is in fact closely related to Theorem \ref{kestenmain}.
\end{remark}
\end{comment}

\section{Higher Order Convergence}\label{section:5}
For a general sequence of probability distributions, the convergence in Wasserstein space $P_p(\mathbb{R})$ is stronger than the weak convergence in $P(\mathbb{R})$. However, if the supports of the sequence of probability distributions are uniformly bounded, the weak convergence implies the convergence in Wasserstein space $P_p(\mathbb{R})$ for any $p \in [1,\infty)$, but not for the case $p=\infty$. For example, the spectral measures of a sequence of regular graphs with fixed vertex degree are in this case.

For a sequence of regular graphs with unbounded vertex degree, the corresponding spectral measures $\mu_n$ do not have uniformly bounded supports. A natural question is when \emph{higher order} convergence  in $P_p(\mathbb{R})$ of such $\mu_n$ holds for large $p\in [1,\infty]$.

\subsection{Some Examples of $W_{\infty}$ Convergence}
As mentioned in Remark \ref{betaJacobi}, the Kesten-McKay distribution converges weakly to the semicircle distribution as the degree tends to infinity. In fact the convergence in $P_{\infty}(\mathbb{R})$ holds, as shown in the proposition below.

\begin{proposition}
Let $\mu_q$ be the Kesten-McKay distribution, i.e., the normalized spectral measure of a $(q+1)\text{-}$regular tree. Let $\mu_{\infty}$ be the semicircle distribution. Then, the following holds as $q\rightarrow \infty$:\\
(i) $\mu_q \rightarrow \mu_{\infty}$ in $P_{\infty}(\mathbb{R})$.\\
(ii) $F_{\mu_q}^{-1} \rightarrow F_{\mu_{\infty}}^{-1}$ uniformly on $(0,1)$.
\end{proposition}
See Figure \ref{figure:comparison1} for a comparison of the Kesten-McKay distribution $\mu_{50}$ and the semicircle distribution $\mu_{\infty}$.
\begin{figure}
    \centering
    \begin{tikzpicture}
\begin{axis}[xmin=-2.5,xmax=2.5,ymin=-0.1,ymax=0.5,
    axis lines = middle,
    ytick={0},
    yticklabels={},
    xtick={-2, 0, 2},
    xlabel = \(x\),
    ylabel = {\(\rho(x)\)}
]
\addplot[
    domain=-2:2,
    samples=100, 
    color=green]{(4-x^2)^(0.5)/(2*pi)};
\addlegendentry{\(\mu_{\infty}\)}
\addplot[
    domain=-2:2,
    samples=1000,
    color=blue]{51*50*(4-x^2)^(0.5)/(2*pi*(51*51-50*x*x))};
\addlegendentry{\(\mu_{50}\)}

\end{axis}
\end{tikzpicture}
    \caption{Comparision of the Kesten-McKay distribution $\mu_{50}$ and the semicircle distribution $\mu_{\infty}$}
    \label{figure:comparison1}
\end{figure}
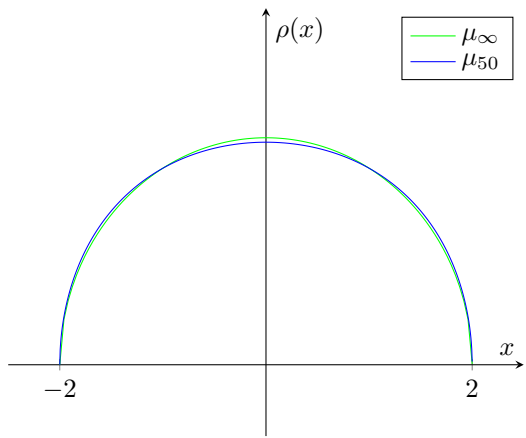

\begin{proof}
    First notice that as $q \rightarrow \infty$ the density function $$\rho_q(x):=\frac{1}{2\pi}\frac{(q+1)\sqrt{4-x^2}}{ (q^{-1/2}+q^{1/2})^2-x^2}\mathbf{1}_{|x|\leq 2}$$ converges uniformly to $$\rho_{\infty}(x):=\frac{1}{2\pi}\sqrt{4-x^2}\mathbf{1}_{|x|\leq 2},$$ 
    since an easy calculation shows $|\rho_q-\rho_{\infty}|\leq 2/(q-2)$ for $q\geq 3$.
     For any $p\in (0,1)$, let $F^{-1}_{\mu_q}(p)=x_1$ and $F^{-1}_{\mu_{\infty}}(p)=x_2$. Then we have by definition
    $$\int_{-2}^{x_1}\rho_q(x) dx=\int_{-2}^{x_2}\rho_{\infty}(x) dx=p,$$
    and hence $$\int_{x_1}^{x_2}\rho_{\infty}(x) dx=\int_{-2}^{x_1}(\rho_{q}(x)-\rho_{\infty}(x))dx.$$
    Therefore, we obtain $$\rho_{\infty}\left(\frac{x_1+x_2}{2}\right)|x_1-x_2|\leq \left|\int_{x_1}^{x_2}\rho_{\infty}(x) dx\right|\leq \int_{-2}^{2}|\rho_{q}(x)-\rho_{\infty}(x)|dx,$$
    where the first inequality comes from the concavity of $\rho_{\infty}$. If $x_1+x_2<0$, then $$\rho_{\infty}\left(\frac{x_1+x_2}{2}\right)\geq \frac{1}{\pi}\left(\frac{x_1+x_2}{2}+2\right)\geq \frac{1}{2\pi}|x_1-x_2|.$$ 
    Similarly, when $x_1+x_2\geq0$ we have $$\rho_{\infty}\left(\frac{x_1+x_2}{2}\right)\geq \frac{1}{\pi}\left(2-\frac{x_1+x_2}{2}\right)\geq \frac{1}{2\pi}|x_1-x_2|.$$ In either case, we derive
    $$\frac{1}{2\pi}|x_1-x_2|^2\leq \int_{-2}^{2}|\rho_{q}(x)-\rho_{\infty}(x)|dx.$$
    Hence $F^{-1}_{\mu_q}$ converges uniformly to $F^{-1}_{\mu_{\infty}}$ on $(0,1)$ as $q \rightarrow \infty$. 
    This proves (ii). Then (i) follows directly due to Lemma \ref{closedformwasserstein}.
\end{proof}

The following proposition provides another example of $W_{\infty}$ convergence of spectral measure of circles.
See Figure \ref{figure:comparison2} for a comparison of the IDF of the spectral measure of circle $C_{53}$ and the $2$-regular tree.

\begin{proposition}
   Let $\mu(C_m)$ be the spectral measure of the $m\text{-}$circle $C_m$. Then $\mu(C_m)\rightarrow \mu_1$ in $P_{\infty}(\mathbb{R})$ as $m\rightarrow \infty$. 
\end{proposition}

\begin{proof}
Firstly, it is direct to check that $$F^{-1}_{\mu_1}(p)=-2\cos \pi p, \ p\in[0,1].$$
Denote the IDF of $\mu(C_m)$ as $F^{-1}_m$ for short. For an even number $m$,
    $$F^{-1}_{m}(p)=-2, \ p \in \left(0,\frac{1}{m}\right),$$
    and 
    $$F^{-1}_{m}(p)=-2\cos \frac{2k\pi }{m}, \ p\in \left(\frac{2k-1}{m}, \frac{2k+1}{m}\right).$$
   For an odd number $m$, 
    $$F^{-1}_{m}(p)=-2\cos \frac{\pi}{m}, \ p \in (0,1/m),$$
    and for $1\leq k \leq m/2$,
    $$F^{-1}_{m}(p)=-2\cos \frac{(2k+1)\pi}{m}, \ p\in \left(\frac{2k-1}{m}, \frac{2k+1}{m}\right).$$ 
    Since $|\cos \alpha -\cos \beta|\leq |\alpha -\beta|$, we have
    $$|F^{-1}_{\mu_1}(p)-F^{-1}_{m}(p)|\leq \frac{4 \pi }{m}\ \text{a.e. on}\ [0,1] .$$
    This completes the proof.
\end{proof}

\begin{figure}
    \centering
    \begin{tikzpicture}{a}
\begin{axis}[
    axis lines = middle,
    xlabel = \(p\),
    ylabel = {\(F^{-1}_{\mu}(p)\)},
    xmin=-0.1, xmax=1.1,
    xtick={0,1/2,1},
    xticklabels={$0$, $\frac{1}{2}$, $1$},
    xticklabel style={anchor=north},
    ymin=-2.5, ymax=3.9,
    ytick={-2,0,2},
    yticklabels={$-2$, $0$, $2$},
    yticklabel style={anchor=north east},
    ymajorgrids=true,
    grid style=dashed
]
\addplot[
    domain=0:1, 
    samples=100,
    color=green]{(-2)*cos((pi)*deg(x))};
    \addlegendentry{\(\mu_1\)};
    \addplot[domain=0/53:1/53,color=blue]{(-2)*cos((pi)*deg(1/53))};
\addplot[domain=1/53:3/53,color=blue]{(-2)*cos((pi)*deg(3/53))};
\addplot[domain=3/53:5/53,color=blue]{(-2)*cos((pi)*deg(5/53))};
\addplot[domain=5/53:7/53,color=blue]{(-2)*cos((pi)*deg(7/53))};
\addplot[domain=7/53:9/53,color=blue]{(-2)*cos((pi)*deg(9/53))};
\addplot[domain=9/53:11/53,color=blue]{(-2)*cos((pi)*deg(11/53))};
\addplot[domain=11/53:13/53,color=blue]{(-2)*cos((pi)*deg(13/53))};
\addplot[domain=13/53:15/53,color=blue]{(-2)*cos((pi)*deg(15/53))};
\addplot[domain=15/53:17/53,color=blue]{(-2)*cos((pi)*deg(17/53))};
\addplot[domain=17/53:19/53,color=blue]{(-2)*cos((pi)*deg(19/53))};
\addplot[domain=19/53:21/53,color=blue]{(-2)*cos((pi)*deg(21/53))};
\addplot[domain=21/53:23/53,color=blue]{(-2)*cos((pi)*deg(23/53))};
\addplot[domain=23/53:25/53,color=blue]{(-2)*cos((pi)*deg(25/53))};
\addplot[domain=25/53:27/53,color=blue]{(-2)*cos((pi)*deg(27/53))};
\addplot[domain=27/53:29/53,color=blue]{(-2)*cos((pi)*deg(29/53))};
\addplot[domain=29/53:31/53,color=blue]{(-2)*cos((pi)*deg(31/53))};
\addplot[domain=31/53:33/53,color=blue]{(-2)*cos((pi)*deg(33/53))};
\addplot[domain=33/53:35/53,color=blue]{(-2)*cos((pi)*deg(35/53))};
\addplot[domain=35/53:37/53,color=blue]{(-2)*cos((pi)*deg(37/53))};
\addplot[domain=37/53:39/53,color=blue]{(-2)*cos((pi)*deg(39/53))};
\addplot[domain=39/53:41/53,color=blue]{(-2)*cos((pi)*deg(41/53))};
\addplot[domain=41/53:43/53,color=blue]{(-2)*cos((pi)*deg(43/53))};
\addplot[domain=43/53:45/53,color=blue]{(-2)*cos((pi)*deg(45/53))};
\addplot[domain=45/53:47/53,color=blue]{(-2)*cos((pi)*deg(47/53))};
\addplot[domain=47/53:49/53,color=blue]{(-2)*cos((pi)*deg(49/53))};
\addplot[domain=49/53:51/53,color=blue]{(-2)*cos((pi)*deg(51/53))};
\addplot[domain=51/53:53/53,color=blue]{(-2)*cos((pi)*deg(53/53))};
\addlegendentry{\(\mu(C_{53})\)}
\end{axis}
\end{tikzpicture}
    \caption{Comparision of the IDF of $\mu(C_{53})$ and that of the arcsine law $\mu_1$}
    \label{figure:comparison2}
\end{figure}
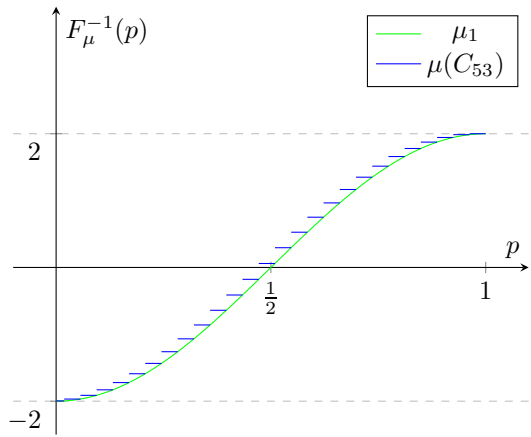

\subsection{Higher order convergence and locally tree-like condition}

For a sequence of random regular graphs $G_n \in G(n,q_n+1)$ with adjacency matrix $A_n$, Dumitriu and Pal \cite{Dumitriu-Pal} show that if $q_n=n^{o(1)}$ and $q_n$ tends to infinity, then the normalized spectral measure of $G_n$ converges weakly to the semicircle distribution in probability. Later, Tran, Vu, and Wang \cite{Tran-Vu-Wang}
extends their result to the case that $q_n \leq n/2$.

In this subsection, we use the Chebyshev polynomials and a generalization of Theorem \ref{kestenmain} to show that the condition that $q_n=n^{o(1)}$ and $q_n$ tends to infinity in fact implies the convergence in Wasserstein space to the semicircle distribution. 
In fact, the condition $q_n=n^{o(1)}$ is also necessary for convergence in $P_p(\mathbb{R})$ for any $p\in [1,\infty)$.
\begin{lemma}
    Let $G_n=(V_n, E_n)$ be a sequence of $(q_n+1)$-regular graphs with $|V_n|\rightarrow \infty$. If the normalized spectral measure $\mu_n$ converges in $P_p(\mathbb{R})$ for any $p\in [1,\infty)$ to $\mu_{\infty}$ then $q_n=|V_n|^{o(1)}$.
\end{lemma}
\begin{proof}
    By Lemma \ref{criterionwasserstein}, for any $r\in \mathbb{N}^+$, $$\int_\mathbb{R} x^{2r}d\mu_n(x) \rightarrow \int_\mathbb{R} x^{2r}d\mu_{\infty}(x) \,\,\text{as}\ n\rightarrow \infty,$$
    the sequence $\{\int x^{2r}d\mu_n(x)\}_n^{\infty}$ is bounded. On the other hand, we have
    $$0<q_n^r/|V_n|\leq \int_\mathbb{R} x^{2r}d\mu_n(x).$$
    This implies $q_n=|V_n|^{o(1)}$.
\end{proof}

However, the condition $q_n=|V_n|^{o(1)}$ does not guarantee in general the convergence of $\mu_n$ to $\mu_{\infty}$. On the other hand, the convergence does hold for almost all such sequences of regular graphs. By using the Chebyshev polynomials, we next strengthen the result of Dumitriu and Pal as follows.
\begin{theorem}\label{Dumitriu-Pal}
    For a sequence of random regular graphs $G_n \in G(n,q_n+1)$ with $q_n =n^{o(1)}$ and $q_n$ tends to infinity, the correspoinding normalized spectral measure $\mu_{n}=\mu(G_n)$ converges to the semicircle distribution $\mu_\infty$ in $(P_{p}(\mathbb{R}),W_{p})$ for all $p\in [1,\infty)$ almost surely.
    \begin{comment}
     (i) For any $p \in [1,\infty)$, $\mu_n \rightarrow \mu_{\infty}$ in $P_p(\mathbb{R})$  as $n\rightarrow \infty$.\\
    (ii) For any $p \in [1,\infty)$, $F_{\mu_n}^{-1} \rightarrow F_{\mu_{\infty}}^{-1}$ in $L^p[0,1]$ as $n\rightarrow \infty$.\\
    \end{comment}
\end{theorem}

In order to prove Theorem \ref{Dumitriu-Pal}, we establish a characterization of convergence in Wasserstein space in the same spirit of Theorem \ref{kestenmain}.
The next lemma shows that $P_{p}(\mathbb{R})$ convergence to the semicircle distribution is equivalent to the lack of small circuits.

\begin{lemma}\label{lackofsmallcircuit}
    Let $G_n=(V_n, E_n)$ be a sequence of $(q_n+1)$-regular graphs. Suppose that the corresponding normalized spectral measure $\mu_n \rightharpoonup \mu_{\infty}$ as $n\rightarrow \infty$. Let $m \in \mathbb{N}^+$ be an even number. Then the following are equivalent:
    \begin{itemize}
        \item [(i)] For any $p \in [1,m]$, $\mu_n \rightarrow \mu_{\infty}$ in $P_p(\mathbb{R})$  as $n\rightarrow \infty$.
    \item[(ii)] For any $ p \in [1,m]$, $F_{\mu_n}^{-1} \rightarrow F_{\mu_{\infty}}^{-1}$ in $L^p[0,1]$ as $n\rightarrow \infty$;
    \item[(iii)] $q_n \rightarrow \infty$ and for $ r \in\{1,\ldots,m\}$, 
$$q_n^{-r/2}\frac{c_{r}(G_n)}{|V_n|} \rightarrow 0\ \text{as}\ n\rightarrow \infty.$$
\item[(iv)] $q_n \rightarrow \infty$ and for $ r\in \{1,\ldots,m\}$, 
$$q_n^{-r/2}\frac{f_{r}(G_n)}{|V_n|} \rightarrow 0\ \text{as}\ n\rightarrow \infty.$$
    \end{itemize}
\end{lemma}

\begin{proof}
    (i)$\Leftrightarrow$(ii): This is due to Lemma \ref{closedformwasserstein}. \\
    (iii)$\Leftrightarrow$(iv): By Lemma \ref{NBWandcircle}, we have 
    $$q^{-r/2}c_r\leq q^{-r/2}f_r \leq q^{-r/2}c_r+\sum_{1\leq i<  r/2}q^{-r/2+i}c_{r-2i},$$
    showing the equivalence of (iii) and (iv). \\
    (i)$\Rightarrow$(iii) 
For $r=2$,
    $$q_n^{-1}\frac{c_{2}(G_n)}{|V_n|}+q_n^{-1}-1=\int  Y_2(x) d\mu_n(x) \rightarrow \int Y_2(x) d\mu_{\infty}(x)=-1,$$
thus as $n \rightarrow \infty$, $q_n \rightarrow \infty$ and
$$q_n^{-1}\frac{c_{2}(G_n)}{|V_n|} \rightarrow 0.$$ 
For  $2 < r \leq m $,
   again $$\int_\mathbb{R} Y_r(x) d\mu_n(x) \rightarrow \int_\mathbb{R} Y_r(x) d\mu_{\infty}(x)=0.$$
Recall 
    $$\int_\mathbb{R}  Y_r(x) d\mu_n(x)=q_n^{-r/2}\frac{c_{r}(G_n)}{|V_n|},$$
    $$\int_\mathbb{R}  Y_r(x) d\mu_n(x)=q_n^{-r/2}\frac{c_{r}(G_n)}{|V_n|}+(1-q_n)q_n^{-r/2},$$
    since for $r >2$, $(1-q_n)q_n^{-r/2}\rightarrow 0$ as $n \rightarrow \infty$,
    we have also $$q_n^{-r/2}\frac{c_{r}(G_n)}{|V_n|} \rightarrow 0.$$\\
(iii)$\Rightarrow$(i)
As seen in the proof of (i)$\Rightarrow$(iii), now for $2\leq r \leq m$,
$$\int_\mathbb{R} Y_r(x) d\mu_n(x) \rightarrow \int_\mathbb{R} Y_r(x) d\mu_{\infty}(x),$$
thus for any polynomial $P$ with $\mathrm{deg}\ P \leq m$,
$$\int_\mathbb{R} P(x) d\mu_n(x) \rightarrow \int_\mathbb{R} P(x) d\mu_{\infty}(x),$$
in particular 
$$\int_\mathbb{R} |x|^m d\mu_n(x)=\int_\mathbb{R} x^m d\mu_n(x) \rightarrow \int_\mathbb{R} x^m d\mu_{\infty}(x)=\int_\mathbb{R} |x|^m d\mu_{\infty}(x).$$
By Lemma \ref{criterionwasserstein} and $\mu_n \rightharpoonup \mu_{\infty}$, we deduce that $\mu_n \rightarrow \mu_{\infty}$ in $P_p(\mathbb{R})$ for any $p \in [1,m]$.
This completes the proof.
\end{proof}

As a result, we have the following characterization.
\begin{theorem}[Locally tree-like $\Leftrightarrow$ Convergence in moments]\label{momentlocallytree}
    For a sequence of $(q_n+1)$-regular graphs $G_n=(V_n, E_n)$, the following are equivalent:
    \begin{itemize}
     \item[(i)] For any $p \in [1,\infty)$, $\mu_n \rightarrow \mu_{\infty}$ in $P_p(\mathbb{R})$ as $n\rightarrow \infty$;
     \item[(ii)] For any $p \in [1,\infty)$, the IDF $F_{\mu_n}^{-1}$ converges to the IDF $F_{\mu_{\infty}}^{-1}$ in $L^p[0,1]$;
     \item[(iii)] The spectral measure $\mu_n$ of $G_n$ converges in moments to the semicircle distribution;
     \item[(iv)] $q_n\rightarrow \infty$, and for any $r \geq 1$,
    $$q_n^{-r/2}\frac{c_{r}(G_n)}{|V_n|} \rightarrow 0\ \text{as}\ n\rightarrow \infty;$$
     \item[(v)] $q_n\rightarrow \infty$, and for any $r \geq 1$,
    $$q_n^{-r/2}\frac{f_{r}(G_n)}{|V_n|} \rightarrow 0\ \text{as}\ n\rightarrow \infty;$$
    \end{itemize}
\end{theorem}
\begin{proof}
    By the argument in Lemma \ref{lackofsmallcircuit}, (i)$\Leftrightarrow$(ii), and (iii)$\Leftrightarrow$(iv)$\Leftrightarrow$(v).\\
    \indent Notice that the semicircle law is compactly supported, thus is uniquely determined by its moments. By Lemma \ref{equivalentconvergence}, convergence in moments is equivalent to convergence in $P_{p}(\mathbb{R})$ for $p\in [1,\infty)$. It implies (i)$\Leftrightarrow$(iii).
\end{proof}

  As a direct corollary, we have the following result by assuming $G_n$ to be deterministic.
\begin{corollary}
    Let $\{G_n\}_{n=0}^{\infty}$ be a sequence of regular graphs. If both the degree $q_n+1$ and girth $g_n$ of $G_n$ tend to $\infty$, then the normalized spectral measure of $G_n$ converges to the semicircle distribution in $P_{p}(\mathbb{R})$ for any $p\in [1,\infty)$.
\end{corollary}

By applying the same argument as in the proof of Theorem \ref{thm:colormckay}, we derive the following corollary.

\begin{corollary}
Let $\{G_n\}_{n=0}^\infty$ be a sequence of regular graphs. Let $\mathcal{G}_{n}$ be a subgroup of $U(N_{n})$ for all $n$. Then the normalized spectral measure $\mu(G_n)$ converges in $P_{p}(\mathbb{R})$ for any $p\in [1,\infty)$ to the semicircle distribution $\mu_{\infty}$ if and only if the measure $\mu(G_n,\sigma_n)$ converges in $P_{p}(\mathbb{R})$ for any $p\in [1,\infty)$ to $\mu_{\infty}$ for any $\mathcal{G}_{n}$-colors $\{\sigma_n\}_{n=0}^\infty$.
\end{corollary}
 
\begin{proof}[Proof of Theorem \ref{Dumitriu-Pal}]
    First notice that $Z_1(G_n)=Z_2(G_n)\equiv 0$. Denote $\lambda_{k,n}:=\frac{1}{2k}q_n^k$ for $k\geq 3$. Under the condition $q_n=n^{o(1)}$, for $k\geq 3$
    there holds
      $$\frac{\mathbb{E}Z_k(G_n)^2}{\lambda_{k,n}^2+\lambda_{k,n}} \rightarrow 1 \,\, \text{as}\,\, n\rightarrow \infty,$$
      by the result of McKay, Wormald, and Wysocka \cite{McKay-Wormald-Wysocka}.
      Therefore, for any fixed $l\geq 1$, $k\geq 3$, there exits some positive constant $C=C(k)$
   such that  $$\mathbb{E} \frac{q_n^{2l} Z^2_k(G_n)}{n^2}  \leq C \frac{q_n^{2k+2l}}{n^2}=C n^{-2+o(1)}.$$ Thus, we obtain that for any fixed $l, k\geq 1$, there holds
    $$\mathbb{E} \sum_{n=1}^{\infty} \frac{q_n^{2l}Z^2_k(G_n)}{n^2} < \infty.$$
    As a consequence, we have by Borel-Cantelli lemma that for any fixed $l, k\geq 1$,
    $$\frac{q_n^{l} Z_k(G_n)}{n} \rightarrow 0\ \text{a.s. as}\ n\rightarrow \infty.$$
    Applying Lemma \ref{NBWandcircle}, we have for any $r\geq 1$,
    $$q_n^{-r/2}\frac{f_{r}(G_n)}{n} \rightarrow 0\ \text{a.s. as}\ n\rightarrow \infty.$$
    By Theorem \ref{momentlocallytree}, we show that $\mu_n$ converges to $\mu_{\infty}$ in $P_{p}(\mathbb{R})$ for any $p\in [1,\infty)$ almost surely.
\end{proof}

    \noindent \textbf{Acknowledgements:} 
This work is supported by the National Key R \& D Program of China 2023YFA1010200.
YG is supported by the National Key R \& D Program of China 2022YFA100740. WL is supported by the Special Funds of the National Natural Science Foundation of China No. 123B2013. SL is supported by the National Natural Science Foundation of China No. 12031017. YG and WL would like to thank Junyan Zhang and Guangyi Zou for useful discussions.

\appendix

\section{Appendix}

\subsection{Proof of Lemma \ref{NBWandcircle}}\label{A}

\begin{proof}[Proof of Lemma \ref{NBWandcircle}]
    First, notice that any closed non-backtracking walk based at $x\in G$ can be decomposed as a non-backtracking walk from $x$ to some vertex $y$, followed by a circuit based at $y$, and then the reverse of the initial non-backtracking walk from $x$ to some vertex $y$. Let the path between $y$ and $x$ be $\gamma_{y,x}=\left(\{v_{j}\}_{j=0}^{i},\{e_{j}\}_{j=0}^{i-1}\right)$ where $v_{0}=y$ and $v_{i}=x$. When $i=0$, this gives the term $c_r$. If $i\geq 1$, there are $q-1$ choices of $v_{1}$ and $(q-1)q^{j-1}$ choices of $v_{j}$ for all $1\leq j\leq i$. These imply (\ref{frcr}).\\
    \indent Now we consider (\ref{frzr}). For any circle of size $r$, we can obtain $2r$ circuits by changing the base vertex and orientation. This implies the left side of (\ref{frzr}).\\
    \indent For the right side of (\ref{frzr}), we notice that if there is a closed non-backtracking walk based at $x\in G$ of length $r$, then $B_r(x)$ contains a circle of size no more than $r$. Here, $B_r(x)$ is the induced subgraph that contains all vertices of distance no more than $r$ from $x$. This implies there exists a $y$ lying on a circle of size less than $r$ such that $x\in B_{r}(y)$. The number of choices of $y$ is no more than $\sum\limits_{1\leq k \leq r} kZ_{k}(G)$, and the number of $B_r(y)$ is no more than $(q+1)q^{r-1}$. We notice that the number of closed non-backtracking walks based at $x$ of length $r$ is no more than $(q+1)q^{r-1}$. Therefore, we obtain that
    $$f_{r}(G)\leq (q+1)q^{r-1}\cdot (q+1)q^{r-1} \cdot \sum_{1\leq k\leq r}kZ_{k}(G).$$
    Thus we prove Lemma \ref{NBWandcircle}.
\end{proof} 

\subsection{Proof of Lemma \ref{Xrqexplicit}}\label{B}

\begin{proof}[Proof of Lemma \ref{Xrqexplicit}]
Consider the Taylor expansion of  $$\frac{1-q^{-1}t^2}{1-xt+t^2}$$ as an analytic function of $t$ near the origin.
We first compute the Taylor expansion of $(1-xt+t^2)^{-1}$. By a change of variable $x=z+z^{-1}$, we have 
$$\frac{1}{1-xt+t^2}=\frac{1}{(z-t)(z^{-1}-t)}=\sum_{r=0}^{\infty}\frac{z^{r+1}-z^{-r-1}}{z-z^{-1}}t^r.$$
Thus we derive
\begin{equation}\label{chebygeneratingfunction}
    \frac{1}{1-xt+t^2}=\sum_{r=0}^{\infty}X_{r}(x)t^{r},
\end{equation}
since $$X_{r}(z+z^{-1})=U_r\left(\frac{z+z^{-1}}{2}\right)=\frac{z^{r+1}-z^{-r-1}}{z-z^{-1}}.$$ 
It is straightforward to check that $$\sum_{r=0}^{\infty}\left(X_{r}(x)-q^{-1}X_{r-2}(x)\right)t^{r}=\frac{1-q^{-1}t^2}{1-xt+t^2}.$$
Therefore, we have (\ref{eq:Xrq}). This completes the proof.
\end{proof}

\begin{comment}
    We present a sketchy proof here for completeness.
\begin{proof}
   
    Using (\ref{eq:Ur}), it is direct to check that $\int X_n(x)X_m(x)d\mu_{\infty}(x)=\delta_{mn}$.
    To show the completeness,
notice that $\mu_{\infty}$ is a Radon measure. Hence $C_c(\mathbb{R})$ is dense in
$L^{2}(\mathbb{R};\mu_{\infty})$, see Folland \cite[Proposition 7.9]{Folland}. Then the completeness is a direct result of Stone-Weierstrass approximation theorem. 
\end{proof}
\end{comment}

\subsection{Proof of Theorem \ref{thm:McKayLaw}}\label{C}
\begin{proof}[Proof of Theorem \ref{thm:McKayLaw}]

To prove Theorem \ref{thm:McKayLaw}, we need the following well-known result of Chebyshev polynomials. 
\begin{lemma}
The polynomials $\{X_{r}\}_{r=0}^{\infty}$ defined in (\ref{chebyshevexplicit}) form a complete orthonormal basis with respect to the semicircle distribution $\mu_{\infty}$.
\end{lemma}

By (\ref{defofkm}) and (\ref{Serrechebyshevmoment}), we have:
$$
\begin{aligned}
    \int_\mathbb{R} X_r d\mu_q&=\langle \mathbf{1}_o, X_r(q^{-1/2}A(\mathbb{T}_q))\mathbf{1}_o\rangle\\
    &=\langle \mathbf{1}_o,  \sum_{0\leq k\leq r/2}q^{-k}X_{r-2k,q}(q^{-1/2}A(\mathbb{T}_q))\mathbf{1}_o\rangle\\
    &=\langle \mathbf{1}_o,  \sum_{0\leq k\leq r/2}q^{-k-(r-2k)/2}A_{r-2k}(\mathbb{T}_q)\mathbf{1}_o\rangle\\
    &=\langle \mathbf{1}_o,  \sum_{0\leq k\leq r/2}q^{-r/2}A_{r-2k}(\mathbb{T}_q)\mathbf{1}_o\rangle.
\end{aligned}
$$
Hence, we have by the definition of non-backtracking matrices that
\[\int_\mathbb{R} X_r d\mu_q=\left\{
\begin{aligned}
    &q^{-r/2}, \ &\text{if $r$ is even};\\
    &0, \ &\text{if $r$ is odd}.
\end{aligned} \right.\]
 For the case $q>1$, we notice that $\sum_{r=0}^{\infty}q^{-r}X_{2r}$ converges on $[-2,2]$ since $|X_r(x)|\leq r+1$ when $x \in [-2,2]$.
Thus we have,
$$d\mu_{q}=\sum_{r=0}^{\infty}q^{-r}X_{2r}\ d\mu_{\infty},$$ by the uniqueness of $\mu_{q}$.
For $|x|\leq 2$, $1/(1-xt+t^2)$  is holomorphic in $|t|<1$. By (\ref{chebygeneratingfunction}), we derive
$$
\begin{aligned}
    \sum_{r=0}^{\infty}q^{-r}X_{2r}(x)&=\frac{1}{2}\left(\frac{1}{1-q^{-1/2}x+q^{-1}}+\frac{1}{1+q^{-1/2}x+q^{-1}}\right)\\
    &=\frac{q+1}{(q^{-1/2}+q^{1/2})^2-x^2}.
\end{aligned}
$$
For the case $q=1$, it is direct to check that the following distribution
$$d\nu(x)=\frac{1}{\pi}\frac{1}{\sqrt{4-x^2}}\mathbf{1}_{|x|\leq 2}\ dx$$
satisfies $\int_{\mathbb{R}} X_r d\nu=1$ when $r$ is even, and  $\int_{\mathbb{R}} X_r d\nu=0$ when $r$ is odd. Again by the uniqueness, we have $\nu=\mu_1$.\\
By the definition of the non-backtracking operators,
\[\langle A_m(\mathbb{T}_q) \mathbf{1}_o,  A_n(\mathbb{T}_q) \mathbf{1}_o\rangle=\left\{
\begin{aligned}
    &(q+1)q^{n-1}, \ &\text{if $m=n\geq 1$};\\
    &0, \ &\text{if $m\neq n$}.
\end{aligned} \right.\]

Due to Lemma \ref{ArandXrq}, we obtain
$$\begin{aligned}
    \int_\mathbb{R} X_{n,q}X_{m,q}d\mu_q&=\langle \mathbf{1}_o, X_{n,q}(q^{-1/2}A(\mathbb{T}_q))X_{m,q}(q^{-1/2}A(\mathbb{T}_q))\mathbf{1}_o\rangle\\
    &=q^{-(m+n)/2}\langle \mathbf{1}_o, A_n(\mathbb{T}_q)A_m(\mathbb{T}_q)\mathbf{1}_o\rangle\\
    &=q^{-(m+n)/2}\langle A_m(\mathbb{T}_q) \mathbf{1}_o, A_n(\mathbb{T}_q) \mathbf{1}_o\rangle.
\end{aligned}$$
This proves (\ref{orthogonalrelation}). The completeness again results from the Stone-Weierstrass approximation theorem.
\end{proof}

\subsection{Proof of Theorem \ref{kestenmain}}\label{D}
\begin{proof}[Proof of Theorem \ref{kestenmain}]
We discuss the almost surely case and in probability case separately. \\
\ \\
    \emph{Proof of the almost surely case:}
    First note that the normalized spectral measure $\mu_{n}$ is supported on $[-q^{-1/2}-q^{1/2}, q^{-1/2}+q^{1/2}]$, thus 
    $\mu_n$ converges weakly to $\mu_{q}$ if and only if $\mu_n$ converges in moments to $\mu_{q}$. Hence $\mu_n \rightharpoonup \mu_{q}$ if and only if for any $r\in \mathbb{N}$,
    $$\int_\mathbb{R} X_{r,q}(x)d\mu_{n}(x) \rightarrow \int_\mathbb{R} X_{r,q}(x)d\mu_{q}(x)\ \text{as}\ n\rightarrow \infty.$$
    Recall $$\int_\mathbb{R} X_{r,q}(x)d\mu_{n}(x)= \int_\mathbb{R} X_{r,q}(x)d\mu_{\infty}(x)\equiv 1,$$
    and $$\int_\mathbb{R}  X_{r,q}(x)d\mu_{n}(x)=q^{-r/2}\frac{f_r(G_n)}{|V_n|},$$
    $$\int_\mathbb{R}  X_{r,q}(x)d\mu_{q}(x)=0,$$
    we see $\mu_n \rightharpoonup \mu_{q}$ if and only if for any $r\geq 1$, $$\frac{f_{r}(G_n)}{|V_n|} \rightarrow 0\ \text{as}\ n\rightarrow \infty.$$ 
    In this way, for a sequence of random graph, $\mu_n \rightharpoonup \mu_{q}$ a.s. if and only if for any $r\geq 1$, $$\frac{f_{r}(G_n)}{|V_n|} \rightarrow 0 \ \text{a.s. as}\ n\rightarrow \infty.$$
    The equivalence of (ii) and (iii) is a direct result of (\ref{frzr}) of Lemma \ref{NBWandcircle}.\\
    \ \\
    \emph{Proof of the in probability case:}
    Choose $\chi \in C(\mathbb{R})$, $|\chi| \leq 1$ such that $\chi(x) \equiv 1, |x|\leq q^{-1/2}+q^{1/2}$ and $\chi(x) \equiv 0, |x|\geq q^{-1/2}+q^{1/2}+1$. \\ \ \\
 (i)$\Rightarrow$(ii): $\chi X_{r,q} \in C_b(\mathbb{R})$, by definition 
 $$\int_\mathbb{R} X_{r,q} d\mu_n=\int_\mathbb{R} \chi X_{r,q} d\mu_n\rightarrow \int_\mathbb{R} \chi X_{r,q}d\mu_q=\int_\mathbb{R} X_{r,q}d\mu_q=0,$$
 in probability for $r\geq 1$. Recall
 $$\int_\mathbb{R} X_{r,q} d\mu_n=\frac{1}{|V_n|}\mathrm{Tr}( X_{r,q}(q^{-1/2}A_n))=q^{-r/2}\frac{f_r(G_n)}{|V_n|},$$ the proof is completed.\\ \ \\
 (ii)$\Rightarrow$(i): 
 (ii) implies that for any $r\geq 0$, $\int_\mathbb{R} X_{r,q} d\mu_n \rightarrow \int_\mathbb{R} X_{r,q} d\mu_q$ in probability. (Notice that $\int_\mathbb{R} X_{0,q} d\mu_n\equiv 1$ is deterministic.) Thus for any polynomial $P$, $\int_\mathbb{R} P d\mu_n \rightarrow \int_\mathbb{R} P d\mu_q$ in probability. Fix $f\in C_b(\mathbb{R})$, We need to show for any $\varepsilon >0$,
    \begin{equation*}
        \mathbf{P}\left(\left|\int_\mathbb{R} fd\mu_n-\int_\mathbb{R} fd\mu_{q}\right| \geq \varepsilon\right) \rightarrow 0\ \text{as}\ n\rightarrow \infty,
    \end{equation*}
which is equivalent to 
 \begin{equation*}
        \mathbf{P}\left(\left|\int_\mathbb{R} \chi fd\mu_n-\int_\mathbb{R} \chi fd\mu_{q}\right| \geq \varepsilon\right) \rightarrow 0\ \text{as}\ n\rightarrow \infty.
    \end{equation*}
Now $\chi f\in C_c(\mathbb{R})$, from the Stone-Weierstrass there exits polynomial $P_{\varepsilon}$ such that $|\chi f-P_{\varepsilon}|< \varepsilon/3$ on $[-q^{-1/2}-q^{1/2}-1, q^{-1/2}+q^{1/2}+1]$. Since $\mu_n$ and $\mu_{q}$ are probability measures, we have $|\int_\mathbb{R} \chi fd\mu_n-\int_\mathbb{R} P_{\varepsilon}d\mu_{n}|< \varepsilon/3$ and $|\int_\mathbb{R} \chi fd\mu_q-\int_\mathbb{R} P_{\varepsilon}d\mu_{q}|< \varepsilon/3$. This leads to
    \begin{equation*}
        \mathbf{P}\left(\left|\int_\mathbb{R} \chi fd\mu_n-\int_\mathbb{R} \chi fd\mu_{q}\right| \geq \varepsilon\right) \leq \mathbf{P}\left(\left|\int_\mathbb{R} P_{\varepsilon}d\mu_n-\int_\mathbb{R} P_{\varepsilon} d\mu_{q}\right|\geq \varepsilon/3\right).
    \end{equation*}
    The right hand side of the above inequality tends to $0$ as $n$ tends to infinity.
\\ \ \\
(ii)$\Leftrightarrow$(iii) It is again due to (\ref{frzr}) of Lemma \ref{NBWandcircle}.

\end{proof}

\bibliographystyle{abbrv}

\end{document}